\documentclass[12pt]{article} 

\usepackage{amssymb}
 \usepackage{amsmath,amsthm}
\usepackage{enumerate}
\usepackage{graphicx}
\newtheorem{thm}{Theorem}[section]
\newtheorem{cor}[thm]{Corollary}

\theoremstyle{definition}

\newtheorem{rem}[thm]{Remark}

\numberwithin{equation}{section}
 \usepackage{lineno}

\begin{document}




\title{Spectra of Bochner-Riesz means on $L^p$}

\author{Yang Chen\\
College of Mathematics and
Computer Science\\
Hunan Normal University\\
Hunan 410081, P. R. China\\
E-mail: yang\_chen0917@163.com
\and
Qiquan Fang\\
School of Science\\
Zhejiang University of Science and Technology\\
Hangzhou, Zhejiang 310023, P. R. China\\
E-mail: fendui@yahoo.com
\and
Qiyu Sun\\
Department of Mathematics\\
University of Central Florida\\
Orlando, FL 32816, USA\\
E-mail: qiyu.sun@ucf.edu}

\date{\today}

\maketitle


\begin{abstract}
The Bochner-Riesz means are shown to have either the unit interval $[0,1]$ or the whole complex plane as their spectra on $L^p, 1\le p<\infty$

\end{abstract}

%
%
%



\section{Introduction and Main Results}
Define Fourier transform $\hat f$ of
an integrable function $f$ by
$$\hat{f}(\xi):=\int_{\mathbb{R}^{d}}f(x)e^{-ix\xi}dx$$
and extend its definition to all tempered distributions as usual.
Consider Bochner-Riesz  means $B_\delta, \delta>0$, on $\mathbb{R}^{d}$,
\begin{equation*}
\widehat{B_{\delta}f}(\xi):=
(1-|\xi|^{2})^{\delta}_{+}\hat{f}(\xi),
\end{equation*}
where $t_+:=\max (t, 0)$ for $t\in \mathbb{R}$ \cite{bochner36, grafakosbook,   steinbook}.
A famous conjecture in Fourier analysis  
 is that
  the Bochner-Riesz mean $B_{\delta}$ is bounded on $L^p:=L^p(\mathbb {R}^d)$, the space
 of all $p$-integrable functions on $\mathbb{R}^{d}$ with its norm denoted by $\|\cdot\|_p$, if and only if
\begin{equation}\label{criticalexponent.eq}
\delta> 
\left(d\left|\frac{1}{p}-\frac{1}{2}\right|-\frac{1}{2}\right)_+, \ 1\le p<\infty.\end{equation}
The requirement \eqref{criticalexponent.eq} on the index $\delta$ is necessary for
$L^p$ boundedness of the Bochner-Riesz mean $B_\delta$ \cite{herz54}.
The sufficiency
is completely solved only for dimension two \cite{carlesonsjolin} and it is still open for high dimensions,
see \cite{bourgain91, lee04, stein58, taovargas00, taovargas00b, taovargasvega98, wolff95} and references therein
for recent advances.

 Denote the identity operator by $I$. For $\lambda\not \in [0, 1]$, we first show that Bochner-Reisz means $B_\delta$ is bounded on $L^p$ if and only if
its resolvents $(z I-B_\delta)^{-1}$, which are multiplier operators with symbols $(z- (1-|\xi|^{2})^{\delta}_{+})^{-1}$,
 are bounded on $L^p$ for all  $z\in \mathbb{C}\backslash  [0,1]$.

\begin{thm}\label{maintheorem} Let $\delta>0$ and $1\le p<\infty$.
 Then the following statements are equivalent to each other.
 \begin{itemize}

 \item[{\rm (i)}] The Bochner-Riesz mean $B_\delta$ is bounded on $L^p$.

 \item[{\rm (ii)}] $(z I-B_{\delta})^{-1}$ is  bounded on $L^p$ for all $z\in \mathbb{C}\backslash  [0,1]$.

 \item[{\rm (iii)}]  $(z_0 I-B_{\delta})^{-1}$ is  bounded on $L^p$ for some $z_0\in \mathbb{C}\backslash  [0,1]$.
 \end{itemize}
\end{thm}

 It is obvious that $\lambda I-B_\delta, \lambda\in [0, 1]$, does not have bounded inverse on $L^2$,
 as it is a multiplier with symbol $\lambda-(1-|\xi|^2)^\delta_+$.
  In the next theorem, we show that $\lambda I-B_\delta, \lambda\in [0, 1]$, does not have bounded inverse on $L^p$ for all $1\le p<\infty$.

\begin{thm}\label{maintheorem1} Let $\delta>0,  1\le p<\infty$ and $\lambda\in [0,1]$.
Then
$$\inf_{\|f\|_p=1} \|(\lambda I-B_\delta) f\|_p=0.$$
\end{thm}

For any $\delta>0$,
define the spectra of Bochner-Riesz mean  $B_{\delta}$ on $L^p$ by
$$\sigma_p(B_\delta):=\mathbb{C}\backslash \{z\in \mathbb{C}, \ zI-B_\delta \ {\rm has\ bounded\ inverse\ on} \ L^p\}.$$
As  Bochner-Riesz means $B_\delta$ are multiplier operators
with symbols $(1-|\xi|^2)_+^\delta$, we have
\begin{equation*}
\sigma_p(B_\delta)={\rm closure\ of}\ \{(1-|\xi|^2)_+^\delta, \ \xi\in \mathbb{R}^d\}=[0,1] \ {\rm for} \ p=2.
\end{equation*}

The equivalence in Theorem \ref{maintheorem} and  stability in Theorem \ref{maintheorem1} may not help to solve the conjecture on Bochner-Riesz means, but they imply that for any $\delta>0$,
the spectra
of  the Bochner-Riesz mean  $B_{\delta}$ on $L^p$
is invariant for different $1\le p<\infty$ whenever it is bounded on $L^p$.

\begin{thm}\label{maincor1} Let $\delta>0$ and $1\le p<\infty$. Then
\begin{itemize}
\item[{(i)}]$\sigma_p(B_\delta)=[0,1]$ if the Bochner-Riesz mean  $B_{\delta}$ is bounded on $L^p$; and

\item[{(ii)}] $\sigma_p(B_\delta)=\mathbb{C}$ if   $B_{\delta}$ is unbounded on $L^p$.
\end{itemize}
\end{thm}

The above spectral invariance  on  different $L^p$ spaces
 holds for any multiplier operator $T_m$ with its bounded symbol $m$ satisfying the following hypothesis,
 \begin{equation}\label{mikhlin.req}
 |\xi|^k |\nabla^k m(\xi)|\in L^\infty,\  0\le k\le d/2+1,\end{equation}
  in  the classical
Mikhlin multiplier theorem,
because in this case,
$$\sigma_2(T_m)={\rm closure\ of} \ \{ m(\xi), \ \xi\in \mathbb{R}^d\},$$
and for any $z\not\in \sigma_2(T_m)$, the inverse of  $z I-T_m$ is a multiplier operator with
symbol  $(z-m(\xi))^{-1}$ satisfying  \eqref{mikhlin.req} too.
Inspired by the above  spectral invariance for  Bochner-Riesz means and Mikhlin multipliers,
 we propose the following problem:
 Under what conditions on symbol  of a multiplier,
does  the corresponding operator have its spectrum on $L^p$ independent on $1\le p<\infty$.

Spectral invariance for different function spaces is  closely related to algebra of singular integral operators \cite{calderon56, farrell10, kurbatov01, sunacha08}
and Wiener's lemma for infinite matrices \cite{grochenigtams06, sunca11, suntams07}.
   It has been established for singular integral operators with kernels being H\"older continuous and having certain off-diagonal decay
\cite{barnes90, farrell10, fss14, ssjfa09, sunacha08}, but it is not well studied yet for Calderon-Zygmund operators, oscillatory integrals, 
 and  many other linear operators  in Fourier analysis.

\smallskip

In this paper, we denote by ${\mathcal S}$  and ${\mathcal D}$
 the space of Schwartz functions and compactly supported $C^\infty$ functions respectively, and we use
 the capital letter $C$ to denote
   an absolute constant that could be different at each occurrence.

\section{Proof of Theorem \ref{maintheorem}}
\label{proof.section}

 Given nonnegative integers $\alpha_0$ and $\beta_0$, let
${\mathcal S}_{\alpha_0, \beta_0}$ contain all functions $f$
with  
$$\|f\|_{{\mathcal S}_{\alpha_0, \beta_0}}
:=\sum_{|\alpha|\le \alpha_0, |\beta|\le \beta_0}\|x^\alpha \partial^\beta f(x) \|_\infty<\infty.$$
In this section, we prove the following strong version of Theorem \ref{maintheorem}. 

\begin{thm}\label{theorem1}
 Let  ${\mathbf B}$ be a Banach space of tempered distributions
 with $\mathcal{S}$ being dense in ${\mathbf B}$. 
Assume that
there exist nonnegative integers $\alpha_0$ and $\beta_0$ such that any convolution operator with kernel $K\in {\mathcal S}_{\alpha_0, \beta_0}$
is bounded on ${\mathbf B}$,  
\begin{equation}\label{space.property1} \|K\ast f\|_{\mathbf{B}}\leq C\|K\|_{{\mathcal S}_{\alpha_0, \beta_0}}\|f\|_{\mathbf{B}} \ \ {\rm for \ all}\  f\in \mathbf{B}.
\end{equation}
 Then the following statements are equivalent to each other.
 \begin{itemize}

 \item[{\rm (i)}] The Bochner-Riesz mean $B_\delta$ is bounded on ${\mathbf B}$.

 \item[{\rm (ii)}] $(z I-B_{\delta})^{-1}$ is  bounded on $\mathbf{B}$ for all $z\in \mathbb{C}\backslash  [0,1]$.

 \item[{\rm (iii)}]  $(z_0 I-B_{\delta})^{-1}$ is  bounded on $\mathbf{B}$ for some $z_0\in \mathbb{C}\backslash  [0,1]$.
 \end{itemize}
\end{thm}

\begin{proof} 
 (i)$\Longrightarrow$(ii).\ Take $z\in \mathbb{C}\backslash  [0,1]$ and
 \begin{equation}\label{theorem1.pf.eq1} r_{0}\in (0, \min(|z/2|^{1/\delta}, 1)/2).\end{equation}
 Let $\psi_1$  and $\psi_2\in {\mathcal D}$ satisfy
\begin{equation} \label{theorem1.pf.eq2}
\psi_{1}(\xi)=1\ {\rm  when} \  |\xi|\leq1-r_{0},\
 \psi_{1}(\xi) =0\ {\rm  when} \   |\xi|\geq1-r_{0}/2;
\end{equation}
and
\begin{equation} \label{theorem1.pf.eq3}
\psi_{2}(\xi)=1-\psi_1(\xi) \ {\rm  if}\  |\xi|\leq 1+r_{0}/2, \
\psi_{2}(\xi)=0\ {\rm  if} \  |\xi|>1+r_{0}.
\end{equation}

Define
$m(\xi):=(z-(1-|\xi|^{2})^{\delta}_{+})^{-1}$, $m_1(\xi):=m(\xi) \psi_1(\xi)$
 and $m_2(\xi):=m(\xi) \psi_2(\xi)$.
Then
$m(\xi)$ is the symbol of the multiplier operator $(z I-B_{\delta})^{-1}$
and
\begin{equation*}
m(\xi)=m_1(\xi)+m_2(\xi)+ z^{-1}(1-\psi_1(\xi)-\psi_2(\xi)).
\end{equation*}
As $m_1, \psi_1, \psi_2\in {\mathcal D}$,
 multiplier operators with symbols $m_1$ and $ \psi_1+\psi_2$
are bounded  on ${\bf B}$ by
 \eqref{space.property1}.
Therefore the proof reduces to  establishing the boundedness of the multiplier operator with symbol $m_2$,
\begin{equation}\label{m2.eq} \|(m_2\hat f)^\vee\|_{\mathbf B}\le C \|f\|_{\mathbf B}, \ f\in {\mathbf B},
\end{equation}
where $f^\vee$ is the inverse Fourier transform of $f$.

Take an integer $N_{0}>\alpha_0/\delta$.
Write
\begin{equation*}
m_{2}(\xi) =z^{-1} \Big(\sum^{N_{0}}_{n=0}+\sum_{n=N_0+1}^\infty\Big)(z^{-1})^{n}
((1-|\xi|^{2})^{\delta}_{+})^{n}\psi_{2}(\xi)
=:m_{21}(\xi)+m_{22}(\xi),
\end{equation*}
and denote  multiplier operators  with symbols $m_{21}$ and $m_{22}$ by $T_{21}$ and $T_{22}$ respectively.
Observe that
$$T_{21}=z^{-1} \Psi_2+\sum_{n=1}^{N_0}z^{-n-1} (B_\delta)^n\Psi_2 ,$$
where $\Psi_2$ is the multiplier operator with symbol $\psi_2$.
Then $T_{21}$ is bounded on ${\bf B}$ by  \eqref{space.property1}
and  the boundedness  assumption (i), 
\begin{equation}\label{bounded.condition2}
\|T_{21} f\|_{\mathbf B}\le C \|f\|_{\mathbf B}\quad {\rm for \ all} \ f\in {\mathbf B}.
\end{equation}

Recall that $\psi_2\in {\mathcal D}$ is supported on $\{\xi, 1-r_0\le |\xi|\le 1+r_0\}$.
 Then  the inverse Fourier transform $K_n$ of $(1-|\xi|^{2})^{n\delta}_{+}\psi_{2}(\xi)$
  satisfies
\begin{equation*}
\| K_n\|_{{\mathcal S}_{\alpha_0, \beta_0}}\le C
n^{\alpha_0} (2r_0)^{n\delta}, \ n\ge N_0+1.
\end{equation*}
 Therefore the convolution  kernel
 $$K(x):=z^{-1}\sum^{\infty}_{n=N_{0}+1}z^{-n}K_{n}(x)$$ 
 of  $T_{22}$ belongs to ${\mathcal S}_{\alpha_0, \beta_0}$.
This together with \eqref{space.property1} proves
\begin{equation}\label{t22.boundedness}
\|T_{22} f\|_{\mathbf B}\le C \|f\|_{\mathbf B}\quad {\rm for \ all} \ f\in {\mathbf B}.\end{equation}

Combining \eqref{bounded.condition2} and \eqref{t22.boundedness} proves  \eqref{m2.eq} and hence  completes the proof of the implication (i)$\Longrightarrow$(ii).

\smallskip

(ii)$\Longrightarrow$(iii).\  The implication is obvious.

\smallskip

(iii)$\Longrightarrow$(i).\  
  Let  $z_0\in \mathbb{C}\backslash  [0,1]$ so that $(z_0 I-B_{\delta})^{-1}$ is  bounded on $\mathbf{B}$,
  $r_0$ be as in \eqref{theorem1.pf.eq1} with $z$ replaced by $z_0$, and let $\psi_1, \psi_2$ be given in \eqref{theorem1.pf.eq2} and \eqref{theorem1.pf.eq3} respectively.
Following the argument used in the proof of the implication (i)$\Longrightarrow$(ii),  we see that it suffices to prove the  operator $T_3$ associated with the multiplier
$m_3(\xi):= (1-|\xi|^2)^\delta_+ \psi_2(\xi)$
is bounded on ${\bf B}$.
Take an integer $N_0>\alpha_0/\delta$ and write
$$m_3(\xi)=-z_0 \Big(\sum_{n=1}^{N_0}+\sum_{n=N_0+1}^\infty\Big) \big(1-z_0 (z_0-(1-|\xi|^2)_+^\delta)^{-1}\big)^n \psi_2(\xi)=:m_{31}(\xi)+m_{32}(\xi),$$
where  the series is convergent since
$$|1-z_0 (z_0-(1-|\xi|^2)_+^\delta)^{-1}|\le \sum_{n=1}^\infty (|z_0|^{-1} (1-|\xi|^2)_+^\delta)^n\le \sum_{n=1}^\infty
(|z_0|^{-1} (2r_0)^\delta)^n
<1.$$
Denote by $T_{31}$ and $T_{32}$ the operators associated with multiplier $m_{31}$ and $m_{32}$ respectively.
As  $T_{31}$ is  a linear combination of $\Psi_2$ and $(z_0 I-B_{\delta})^{-n}\Psi_2, 1\le n\le N_0$,
it is bounded on ${\bf B}$,
\begin{equation}\label{theorem1.pf.100}
\|T_{31} f\|_{\bf B}\le C \|f\|_{\mathbf B} \ \  {\rm for \ all} \ f\in {\bf B},
\end{equation}
by \eqref{space.property1} and  the boundedness assumption (iii).

Define the inverse Fourier transform of $\big(1-z_0 (z_0-(1-|\xi|^2)_+^\delta)^{-1}\big)^n \psi_2(\xi), n>N_0$, by $\tilde K_n$.
One may verify that
\begin{equation*}
\| \tilde K_n\|_{{\mathcal S}_{\alpha_0, \beta_0}}\le C
n^{\alpha_0} (2r_0)^{n\delta} |z_0|^{-n}, \ n\ge N_0+1.
\end{equation*}
Therefore
\begin{equation}\label{theorem1.pf.101}
\|T_{32} f\|_{\bf B}\le  C \sum_{n=N_0+1}^\infty
\| \tilde K_n\|_{{\mathcal S}_{\alpha_0, \beta_0}} \|f\|_{\bf B} \le C \Big(\sum_{n=N_0+1}^\infty n^{\alpha_0} (2r_0)^{n\delta}|z_0|^{-n}\Big)\|f\|_{\bf B}
\end{equation}
for all $f\in {\mathbf B}$. Combining \eqref{theorem1.pf.100} and \eqref{theorem1.pf.101} completes the proof.
\end{proof}

\section{Proof of Theorem \ref{maintheorem1}}
\label{proof.section}
Let $f$  and  $K$ be  Schwartz functions with $f(0)=1$ and $\hat K(0)=0$, and set  $f_N(x)=N^{-d} f(x/N), N\ge 1$.
Then  for any  positive integer $\alpha\ge  d+1$
there exists a constant $C_\alpha$ such that
\begin{eqnarray}\label{cond2.verf}
|K*f_N(x)| & \leq &
\Big(\int_{|x-y|>\sqrt{N}}+\int_{|x-y|\le \sqrt{N}}\Big) |K(x-y)|\ |f_N(y)-f_N(x)| dy\nonumber\\
&\le & \int_{|x-y|> \sqrt{N}} |K(x-y)| \ |f_N(y)| dy+ C_\alpha N^{-d-1/2} (1+|x/N|)^{-\alpha}\nonumber\\
&\le & C_\alpha N^{-1/2}
\Big( \int_{\mathbb{R}^d} (1+|x-y|)^{-\alpha} |f_N(y)| dy+  N^{-d} (1+|x/N|)^{-\alpha}\Big).\nonumber\\
\end{eqnarray}
  This implies that
$$\lim_{N\to \infty}
\frac{\|K*f_N\|_p}{\|f_N\|_p}=0, \ 1\le p<\infty.
$$
Therefore the following is a strong version of Theorem \ref{maintheorem1}.

\begin{thm}\label{theorem2}
 Let  ${\mathbf B}$ be a Banach space of tempered distributions
 with $\mathcal{S}$ being dense in ${\mathbf B}$.
 Assume that \eqref{space.property1} holds for some $\alpha_0, \beta_0\ge 0$ and  that for any $\xi_0\in \mathbb{R}^d$
  there exists $\varphi_0\in {\mathcal D}$ such that
 $\widehat \varphi_0(0)=1$ and
  \begin{equation}\label{theorem2.eq1}
  \lim_{N\to\infty}\frac{\|(m \widehat {f_{N, \xi_0}})^\vee\|_{\mathbf{B}}}
{\|f_{N, \xi_0}\|_{\mathbf{B}}}=0
 \end{equation}
 for  all  Schwartz functions $m$ with $m(\xi_0)=0$, where
 $\widehat {f_{N, \xi_0}}(\xi)=\varphi_0(N(\xi-\xi_0))$.
 Then \begin{equation}
\label{theorem2.pf.eq1}\inf_{f\ne 0} \frac{\|(\lambda I-B_{\delta}) f\|_{\mathbf{B}}}{\|f\|_{\mathbf{B}}}=0 \quad {\rm for \ all} \ \lambda\in [0,1]. \end{equation}
%
\end{thm}
\begin{proof}
The infimum in \eqref{theorem2.pf.eq1} is obvious for $\lambda=0$.
So  we assume that $\lambda\in (0, 1]$ from now on.
%
%
%
 Select $\xi_0\in \mathbb{R}^d$ so that $(1-|\xi_0|^2)^\delta_+=\lambda$.
  Then  for sufficiently large $N\ge 1$,
\begin{equation}
\label{theorem2.pf.eq2}(\lambda I-B_\delta) f_{N, \xi_0}= (m_{\xi_0}\widehat {f_{N, \xi_0}})^\vee,
\end{equation}
  where
$m_{\xi_0}(\xi)=(\lambda-(1-|\xi|^{2})^{\delta}_{+})\psi(\xi-\xi_0)$
and $\psi\in {\mathcal D}$ is so chosen
  that $\psi(\xi)=1$ for $|\xi|\le (1-|\xi_0|)/2$ and $\psi(\xi)=0$ for $|\xi|\geq 1-|\xi_0|$. Observe that $m_{\xi_0}\in {\mathcal D}$ satisfies $m_{\xi_0}(\xi_0)=0$. This together with
 \eqref{theorem2.eq1} and \eqref{theorem2.pf.eq2} proves that
  $$\lim_{N\to\infty}\frac{\|(\lambda I-B_\delta) f_{N, \xi_0} \|_{\mathbf{B}}}
{\|f_{N, \xi_0}\|_{\mathbf{B}}}=0.$$
Hence \eqref{theorem2.pf.eq1} is proved for $\lambda\in (0, 1]$.
\end{proof}

\begin{rem} {\rm For $\xi\in \mathbb{R}^d$,  define  modulation operator $M_\xi$ by
$$M_\xi f(x)= e^{ix \xi} f(x).$$
We say that  a Banach space ${\mathbf B}$ is {\em modulation-invariant} if
  for any $\xi\in \mathbb{R}^d$ there exists a positive constant $C_\xi$ such that
\begin{equation*}
\|M_\xi f\|_{\mathbf B}\le C_\xi \|f\|_{\mathbf B}, \quad f\in {\mathbf B}.
\end{equation*}
Such a Banach space with
modulation bound $C_\xi$ 
 being dominated by a polynomial of $\xi$
 was introduced in \cite{taotams2002}
 to study oscillatory integrals and Bochner-Riesz means.  Modulation-invariant Banach spaces include weighted $L^p$ spaces,  Triebel-Lizorkin spaces  $F_{p, q}^\alpha$,  Besov spaces $B_{p, q}^\alpha$,
  Herz spaces  $K^{\alpha, q}_p$, and
 modulation spaces $M^{p, q}$, where $\alpha\in \mathbb{R}$ and $1\le p, q<\infty$
 \cite{grochenigbook, grafakosbook, luyanghu08, triebelbook}.
 For functions $f_{N, \xi_0}, N\ge 1$, in  Theorem \ref{theorem2},
$$f_{N, \xi_0}(x)= e^{i x\xi_0}  (\varphi_0(N\cdot))^\vee(x)$$
and
$$ (m \widehat {f_{N, \xi_0}})^\vee (x)= e^{i x\xi_0}   (m_{\xi_0}  \varphi_0 (N\cdot))^\vee (x),$$
where $m_{\xi_0}(\xi)=m(\xi+\xi_0)$ satisfies $m_{\xi_0}(0)=0$.
Then for a modulation-invariant space ${\mathbf B}$,
 the limit \eqref{theorem2.eq1}  holds for any $\xi_0\in \mathbb{R}^d$ if and only if it is true for $\xi_0=0$.
Therefore we obtain the following  result from Theorems \ref{theorem1} and  \ref{theorem2}.

\begin{cor}\label{theorem3}
 Let  ${\mathbf B}$ be a modulation-invariant Banach space of tempered distributions
 with $\mathcal{S}$ being dense in ${\mathbf B}$.
 Assume that \eqref{space.property1} holds for some $\alpha_0, \beta_0\ge 0$ and that
  there exists $\varphi_0\in {\mathcal D}$ such that
 $$\hat \varphi_0(0)=1\ {\rm  and} \
  \lim_{N\to\infty} \|(m \varphi_0(N\cdot))^\vee\|_{\mathbf{B}}/
\|(\varphi_0(N\cdot))^\vee\|_{\mathbf{B}}=0$$
 for  all  Schwartz functions $m(\xi)$ with $m(0)=0$.
If the Bochner-Riesz mean  $B_\delta$ is bounded on ${\mathbf B}$, then
its spectrum on ${\mathbf B}$ contains the unit interval
$[0,1]$.
\end{cor}
}\end{rem}

\section{Remarks}
\label{remark.section}

In this section,  we extend  conclusions in   Theorem \ref{maincor1} to weighted $L^p$ spaces, Triebel-Lizorkin spaces, Besov spaces, and Herz spaces.

\subsection{Spectra on weighted $L^p$ spaces}
Let $1\le p<\infty$ and ${\mathcal Q}$ contain all cubes $Q\subset \mathbb{R}^d$. A positive function $w$ is said to be a {\em Muckenhoupt $A_p$-weight}
if
\begin{equation*}
\Big(\frac{1}{|Q|}\int_Q w(x) dx\Big) \ \Big (\frac{1}{|Q|} \int_Q w(x)^{-1/(p-1)} dx\Big)^{p-1}\le C, \ Q\in {\mathcal Q}
\end{equation*}
 for $1<p<\infty$, and
\begin{equation*}
\frac{1}{|Q|} \int_Q w(x) dx\le C \inf_{x\in Q}  w(x), \ Q\in {\mathcal Q}
\end{equation*}
for $p=1$ \cite{garciabook, muckenhoupt72}.
For $\delta>(d-1)/2$, convolution kernel of the Bochner-Riesz mean $B_\delta$ is dominated by a multiple of
$(1+|x|)^{-\delta-(d+1)/2}$ and hence it is bounded on weighted $L^p$ space $L^p_w$ for all $1\le p<\infty$ and Muckenhoupt $A_p$-weights $w$. For $\delta=(d-1)/2$, complex interpolation method was introduced in \cite{shisun92}
to establish $L^p_w$-boundedness of $B_\delta$ for all $1<p<\infty$ and Muckenhoupt $A_p$-weights $w$.
The reader may refer to \cite{garciabook, lisun12} 
and references therein for
$L^p_w$-boundedness of Bochner-Riesz means  with various weights $w$. In this subsection, we consider spectra of Bochner-Riesz means on 
 $L^p_w$.

\begin{thm}\label{weightedtheorem} Let $\delta>0, 1\le p<\infty$,
and $w$ be a Muckenhoupt $A_p$-weight.
If the Bochner-Riesz mean  $B_{\delta}$ is bounded on $L^p_w$, then
its spectrum on $L^p_w$ is the unit interval $[0,1]$.
\end{thm}

\begin{proof} Denote the norm on $L^p_w$ by $\|\cdot\|_{p,w}$.
By Theorems \ref{theorem1} and \ref{theorem2}, and modulation-invariance of  $L^p_w$, it suffices to prove
\begin{equation}\label{weightedtheorem.pf.eq1} \|K\ast f\|_{p,w}\leq C\|K\|_{{\mathcal S}_{d+1, 0}}\|f\|_{p, w} \ \ {\rm for \ all}\  f\in L^p_w,
\end{equation}
and
  \begin{equation}\label{weightedtheorem.pf.eq2}
  \lim_{N\to\infty}\frac{\|(m \varphi_0(N\cdot))^\vee\|_{p,w}}
{\| (\varphi_0(N\cdot))^\vee\|_{p,w}}=0
 \end{equation}
for  all Schwartz functions  $\varphi_0$ and $m$ with
 $\widehat \varphi_0(0)=1$ and $m(0)=0$.

 Observe that
 $|K(x)|\le C\|K\|_{{\mathcal S}_{d+1, 0}}(1+|x|)^{-d-1}$. Then \eqref{weightedtheorem.pf.eq1} follows from the standard argument
 for weighted norm inequalities \cite{garciabook}.

 Recall that any $A_p$-weight is a doubling measure \cite{garciabook}. This doubling property together with
  \eqref{cond2.verf}  leads to
 \begin{equation}\label{weightedtheorem.pf.eq3}
 \|(m \varphi_0(N\cdot))^\vee\|_{p,w}^p\le C N^{-p/2}\|(\varphi_0(N\cdot))^\vee\|_{p, w}^p+ CN^{-(d+1/2)p}
  w([-N, N]^d).
 \end{equation}
 On the other hand, there exists $\epsilon_0>0$ 
  such that
 $|\varphi_0^\vee(x)|\ge |\varphi_0^\vee(0)|/2\ne 0$
 for  all $|x|\le \epsilon_0$.
This implies that
\begin{equation}\label {weightedtheorem.pf.eq4}
\|(\varphi_0(N\cdot))^\vee\|_{p, w}^p\ge C N^{-dp}
w([-\epsilon_0 N, \epsilon_0 N]^d).
\end{equation}
Combining \eqref{weightedtheorem.pf.eq3}, \eqref{weightedtheorem.pf.eq4} and the doubling
property for the weight $w$, we establish the limit
\eqref{weightedtheorem.pf.eq2} and complete the proof.
\end{proof}

\subsection{Spectra on Triebel-Lizorkin spaces and Besov spaces}
Let $\phi_0$ and $\psi\in {\mathcal S}$ be so chosen that $\widehat \phi_0$ is supported in $\{\xi, |\xi|\le 2\}$,
$\hat \psi$ supported in $\{\xi, 1/2\le |\xi|\le 2\}$, and
$$\widehat \phi_0(\xi)+\sum_{l=1}^\infty \hat \psi(2^{-l}\xi)=1,\ \  \xi\in \mathbb{R}^d.$$
For $\alpha\in \mathbb{R}$ and $1\le p, q<\infty$, let Triebel-Lizorkin space $F_{p, q}^\alpha$  contain all tempered distributions $f$
with
$$\|f\|_{F_{p,q}^\alpha}:=\|\phi_0*f\|_p+\Big\|\Big(\sum_{l=1}^\infty
2^{l\alpha q}|\psi_l*f|^q\Big)^{1/q}\Big\|_p<\infty,$$
where $\psi_l =2^{ld} \psi(2^l\cdot), l\ge 1$.
Similarly, let Besov space $B^{\alpha}_{p, q}$  be the space of tempered distributions $f$
with
$$\|f\|_{B_{p,q}^\alpha}:=\|\phi_0*f\|_p+\Big(\sum_{l=1}^\infty
2^{l\alpha q}\|\psi_l*f\|_p^q\Big)^{1/q}<\infty.$$
Next is our results about spectra of Bochner-Riesz means on Triebel-Lizorkin spaces and on Besov spaces.

\begin{thm}
Let $\delta>0, \alpha\in \mathbb{R}$ and  $1\le p, q<\infty$.
If the Bochner-Riesz mean  $B_{\delta}$ is bounded on $F_{p, q}^\alpha$ (resp. $B^\alpha_{p,q}$), then
its spectrum  on $F_{p, q}^\alpha$  (resp. on $B^\alpha_{p, q}$) is the unit interval $[0,1]$.
\end{thm}
\begin{proof}
For $z\not\in [0,1]$ and $\delta>0$,
 both $B_\delta$ and $(z I-B_\delta)^{-1}-z^{-1}I$ are  multiplier operators with compactly supported symbols.
Therefore $B_\delta$ (resp. $(z I-B_\delta)^{-1}$) is bounded on the
Triebel-Lizorkin space $F^\alpha_{p,q}$ if and only if  it is bounded on the Besov space  $B^\alpha_{p,q}$ if and only if it is bounded on $L^p$.
The above equivalence together with Theorem \ref{maintheorem} yields our desired conclusions.
\end{proof}

\subsection{Spectra on Herz spaces}
For $\alpha\in \mathbb{R}$ and $1\le p, q<\infty$, let Herz space $K^{\alpha, q}_{p}$  contain all
locally $p$-integrable functions $f$
with
$$\|f\|_{K_{p}^{\alpha, q}}:=\|f\chi_{|\cdot|\le 1}\|_p+\Big(\sum_{l=1}^\infty
2^{l\alpha q}\|f\chi_{2^{l-1}<|\cdot|\le 2^{l}}\|_p^q\Big)^{1/q}<\infty,$$
where  $\chi_E$ is the characteristic function on a set $E$.
The boundedness of Bochner-Riesz means on Herz spaces is well studied, see  for instance \cite{luyanghu08, wang2011}.
Following the argument used in the proof of Theorem \ref{weightedtheorem}, we have

\begin{thm} Let $\delta>0, 1\le p, q<\infty$ and $\alpha>-d/p$.
If the Bochner-Riesz mean  $B_{\delta}$ is bounded on $K_{p}^{\alpha, q}$, then
its spectrum  on $K_{p}^{\alpha, q}$ is 
$[0,1]$.
\end{thm}

\section*{Acknowledgements}
The project is partially supported by National Science Foundation (DMS-1412413) and NSF of China (Grant Nos. 11426203).


%

\begin{thebibliography}{99}


\bibitem 
{barnes90}
B. A. Barnes, When is the spectrum of a convolution operator on $L_p$ independent of $p$?
Proc. Edinburgh Math. Soc., 33(1990), 327-332.

\bibitem 
{bochner36}
S. Bochner, Summation of multiple Fourier series by spherical means,
Trans. Amer. Math. Soc., 40(1936),  175-207.

\bibitem 
{bourgain91}
J. Bourgain, Besicovitch type maximal operators and applications to Fourier
analysis, Geom. Funct. Anal., 1(1991), 147-187.


\bibitem 
{calderon56}
A. P. Calderon and A. Zygmund, Algebras of certain singular integral operators, Amer. J.  Math.,  78(1956), 310-320.


\bibitem 
{carlesonsjolin}
L. Carleson and P. Sj\"{o}lin,  Oscillatory integrals and a multiplier problem for the
disc, Studia Math., 44(1972), 287-299.

\bibitem 
{farrell10}
B. Farrell and T. Strohmer, Inverse-closedness of a Banach algebra of integral operators on the Heisenberg group, J. Operator Theory, 64(2010), 189-205.

\bibitem 
{fss14}
Q. Fang, C. E. Shin and Q. Sun, Wiener's lemma for singular integral operators of Bessel potential type,
Monat. Math., 173(2014), 35-54.

\bibitem 
{garciabook}
J. Garcia-Cuerva and J. Rubio de Francia, {\em Weighted Norm Inequalities and Related Topics}, North-Holland, Amsterdam,
1985.

\bibitem 
{grochenigbook}
 K. Gr\"ochenig, {\em Foundations of Time-Frequency Analysis}, Birkhauser, 2000.

\bibitem 
{grochenigtams06}
K. Gr\"ochenig and M. Leinert, Symmetry and inverse-closedness of matrix algebras and functional calculus for infinite matrices, Trans. Amer. Math. Soc., 358(2006), 2695-2711.

\bibitem 
{grafakosbook}
L. Grafakos, {\em Classical and Modern Fourier Analysis}, Pearson, 2004.

\bibitem 
{herz54} C. S. Herz, On the mean inversion of Fourier and Hankel transforms, Proc. Nat.
Acad. Sci. U.S.A., 40(1954), 996-999.

\bibitem 
{kurbatov01} V. G. Kurbatov, Some algebras of operators majorized by a convolution, Funct. Differ. Equ., 8(2001), 323-333.

\bibitem 
{lee04}
S. Lee, Improved bounds for Bochner-Riesz and maximal Bochner-Riesz operators, Duke Math. J., 122(2004), 205-232.

\bibitem 
{lisun12}
K. Li and W. Sun,  Sharp bound of the maximal Bochner-Riesz operator in weighted
Lebesgue spaces,  J. Math. Anal. Appl., 395(2012), 385-392.

\bibitem 
{luyanghu08}
S. Lu, D. Yang and G. Hu, {\em Herz Type Spaces and Their Applications},  Science Press, Beijing, 2008.

\bibitem 
{muckenhoupt72}
B. Muckenhoupt, Weighted norm inequalities for the Hardy maximal function,
Trans. Amer. Math. Soc., 165(1972), 207-226.

\bibitem 
{shisun92}
X. Shi and Q. Sun,  Weighted norm inequalities for Bochner-Riesz operators and singular integral operators,
Proc. Amer. Math. Soc., 116(1992), 665-673.

\bibitem 
{ssjfa09}
C. E. Shin and Q. Sun, Stability of localized operators, J. Funct. Anal., 256(2009), 2417-2439.

\bibitem 
{steinbook}
E. M. Stein, {\em Harmonic Analysis: Real-Variable Methods, Orthogonality, and Oscillatory
Integrals}, Princeton University Press, Princeton, 1993.

\bibitem 
{stein58}
E. M. Stein, Localization and summability of multiple Fourier series, Acta Math.,
100(1958), 93-147.


\bibitem 
{sunca11}
Q. Sun, Wiener's lemma for infinite matrices II, Constr. Approx., 34(2011), 209-235.

\bibitem 
{sunacha08}
Q. Sun, Wiener's lemma for localized integral operators, Appl. Comput. Harmonic Anal., 25(2008), 148-167.

\bibitem 
{suntams07}
Q. Sun,  Wiener's lemma for infinite matrices, Trans. Amer. Math. Soc., 359(2007), 3099-3123.

\bibitem 
{triebelbook}
H. Triebel, {\em Theory of Function Spaces}, Birkhauser,  1983;  Theory of Function Spaces II, Birksauser, 1992.

\bibitem 
{taotams2002}
T. Tao,  On the maximal Bochener-Riesz conjecture in the plane for $p<2$, Trans. Amer. Math. Soc.,  354(2002), 1947-1959.

\bibitem 
{taovargas00}
T. Tao and A. Vargas,  A bilinear approach to cone multipliers I: restriction
estimates, Geom. Funct. Anal., 10(2000), 185-215.

\bibitem 
{taovargas00b}
T. Tao and A. Vargas, A bilinear approach to cone multipliers  II: application,
Geom. Funct. Anal., 10(2000), 216-258.

\bibitem 
{taovargasvega98}
T. Tao, A. Vargas and L. Vega, A bilinear approach to the restriction and Kakeya
conjectures, J. Amer. Math. Soc., 11(1998), 967-1000.

\bibitem 
{wang2011}
 H. Wang, Some estimates for Bochner-Riesz operators on the weighted Herz-type Hardy spaces,
 J. Math. Anal. Appl., 381(2011), 134-145.

\bibitem 
{wolff95}
 T. Wolff, An improved bound for Kakeya type maximal functions, Rev. Mat.
Iberoamericana, 11(1995), 651-674.

\end{thebibliography}

\end{document}